\documentclass[pagebackref,colorlinks,citecolor=blue,linkcolor=blue,urlcolor=blue,filecolor=blue]{article}
\pdfoutput=1
\usepackage[margin=1in]{geometry}
\usepackage{tikz-cd}
\usetikzlibrary{calc}
\usepackage{float}
\usepackage{amsmath}
\usepackage{amsfonts}
\usepackage{amsthm}
\usepackage{enumitem}
\usepackage{dsfont}
\usepackage{amssymb}
\usepackage{pifont}
\usepackage{mathtools}
\usepackage{comment}
\usepackage{graphicx, caption} 
\usepackage{float}
\usetikzlibrary{matrix}
\usepackage[T1]{fontenc}
\usepackage[utf8]{inputenc}

\usepackage{pinlabel} 

\usepackage{hyperref}

\newtheorem*{T1}{Theorem~\ref{top complexity of conf(n,w)}}

\newtheorem{thm}{Theorem}[section]
\newtheorem{lem}[thm]{Lemma}
\newtheorem{prop}[thm]{Proposition}

\newtheorem{cor}[thm]{Corollary}
\newtheorem{exam}[thm]{Example}
\theoremstyle{definition}

\theoremstyle{definition}

\newtheorem*{claim*}{Claim}
\newtheorem*{quest*}{Question}
\newtheorem*{remark*}{Remark}
\newtheorem*{fact*}{Fact}

\newcommand{\thmtext}{\[
\textbf{TC}\big(\text{conf}(n,w)\big)=\begin{cases}1&\mbox{if }n=1,\\
2n-2&\mbox{if }1<n\le w,\\
2n-2\big\lceil\frac{n}{w}\big\rceil+1&\mbox{if }n>w.
\end{cases}
\]}

\newcommand{\R}{\ensuremath{\mathbb{R}}}

\title{The topological complexity of the ordered configuration space of disks in a strip}
\author{Nicholas Wawrykow}
\date{}

\begin{document}
\maketitle

\begin{abstract}
How hard is it to program $n$ robots to move about a long narrow aisle such that only $w$ of them can fit across the width of the aisle?
In this paper, we answer that question by calculating the topological complexity of $\text{conf}(n,w)$, the ordered configuration space of open unit-diameter disks in the infinite strip of width $w$.
By studying its cohomology ring, we prove that, as long as $n$ is greater than $w$, the topological complexity of $\text{conf}(n,w)$ is $2n-2\big\lceil\frac{n}{w}\big\rceil+1$, providing a lower bound for the minimum number of cases such a program must consider.
\end{abstract}

\section{Introduction}

How hard is it to program a fleet of $n$ robots to move about a long narrow aisle such that only $w$ of the robots can fit across the width of the aisle?
If we are given the initial and final configurations of the robots, it is rather easy to tell the robots how to move without bumping into each other or the walls.
However, if we seek a universal program that can guide the robots between every possible combination of initial and final configurations and avoids collisions, the programing task becomes far more daunting.
We could write a program that considers every possible relative configuration of the robots, but as $n$ grows the number of cases quickly gets out of hand.
Instead of dealing with such a task, we would like to develop a program that breaks this problem into as few cases as possible.

In this \emph{motion planning problem} we seek to minimize the number of different cases of starting and ending positions of the fleet of robots for which we can find a continuous rule, or \emph{motion planner}.
In the case of $n$ robots in a long aisle such that at most $w$ of the robots fit across the width of the aisle, this problem is equivalent to determining the \emph{topological complexity,} denoted $\textbf{TC}$, of the \emph{ordered configuration space of $n$ open unit-diameter disks in the infinite strip of width $w$}, see Figure \ref{conf52}.
This space, abbreviated $\text{conf}(n,w)$, is the space of ways of embedding $n$ open unit-diameter disks in the infinite strip of width $w$, i.e., it is the following subspace of $\R^{2n}$
\[
\text{conf}(n, w):=\big\{(x_{1}, y_{1}, \dots, x_{n}, y_{n})\in \R^{2n}|(x_{i}-x_{j})^{2}+(y_{i}-y_{j})^{2}\ge 1\text{ for }i\neq j\text{ and }\frac{1}{2}\le y_{i}\le w-\frac{1}{2} \text{ for all } i\big\},
\]
and its topology as has been studied in \cite{alpert2021configuration, alpert2021configuration1, BBK, wawrykow2022On, wawrykow2023representation} among others.
It is a natural restriction of the ordered configuration space of $n$ points in the plane, $F_{n}(\R^{2})$, which Farber and Yuzvinsky proved has topological complexity $2n-2$ if $n>1$ \cite[Theorem 1]{farber2002topological}.

In this paper, we determine $\textbf{TC}\big(\text{conf}(n,w)\big)$ by studying the cohomology ring of $\text{conf}(n,w)$.
This gives us a lower bound for $\textbf{TC}\big(\text{conf}(n,w)\big)$.
This, combined with an identical upper bound arising from a space homotopy equivalent to $\text{conf}(n, w)$, yields the following result.

\begin{thm}\label{top complexity of conf(n,w)}
\thmtext
\end{thm}

It follows that, as long as we have sufficiently many robots, only $w$ of which fit across the aisle, we could never write a program that uses fewer than $2n-2\big\lceil\frac{n}{w}\big\rceil+1$ cases for the pairs of initial and final configurations.
Moreover, if we are very clever, we could write a program with exactly $2n-2\big\lceil\frac{n}{w}\big\rceil+1$ rules.

\begin{figure}[h]
\centering
\captionsetup{width=.8\linewidth}
\includegraphics[width = 8cm]{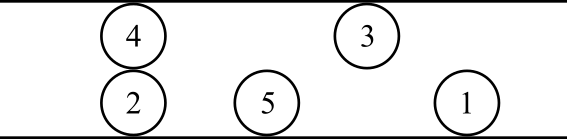}
\caption{A point in $\text{conf}(5,2)$.
Note that two disks can be aligned vertically, but three cannot.
}
\label{conf52}
\end{figure}

\subsection{Acknowledgements}
The author would like to thank Benson Farb for comments on an earlier draft of this paper.

\section{Topological complexity}
Given a space $X$, there is a natural motion planning problem: one would like a motion planner that takes in pair of points $x_{1}, x_{2}\in X$ and produces a continuous path from $x_{1}$ to $x_{2}$.
For example, given two points in $\text{conf}(n,w)$, i.e., two configurations of ordered disks, we would like a rule that tells us how to move from one point to another, i.e., how to move the disks in the first configuration to get the second configuration, keeping the disks in the strip and avoiding collisions.
Moreover, small perturbations to $x_{1}$ or $x_{2}$ should not change the path drastically, i.e., this motion planner should be continuous.
In the case of $\text{conf}(n,w)$, this means that configurations near to the initial and final ones should lead to similar movements of the disks.

We formalize the general problem by letting $X^{I}$ denote the space of all continuous paths 
\[
\gamma:[0,1]\to X,
\]
with the compact-open topology, and write
\begin{align*}
\pi:X^{I}&\to X\times X\\
\pi(\gamma)&=\big(\gamma(0), \gamma(1)\big)
\end{align*}
for the map $\pi$ that associates a path $\gamma\in X^{I}$ to its endpoints.
In this setting the problem of finding a continuous motion planner is equivalent to finding a continuous section of $\pi$, i.e., a map
\[
s:X\times X\to X^{I},
\]
such that $\pi\circ s=\text{id}_{X\times X}$.
If $X$ is convex, such a motion planner exists, namely $s$ takes the points $x_{1}$ and $x_{2}$ to the constant speed straight-line path between them, i.e.,
\[
s(x_{1}, x_{2})(t)=tx_{1}+(1-t)x_{2}.
\]
Unfortunately, continuous motion planners $s$ only exist for slightly more general $X$, as Farber proved that such an $s$ exists if and only if $X$ is contractible \cite[Theorem 1]{farber2003topological}.

Since only the simplest spaces have continuous motion planners, we change our question and ask how bad can the motion planning problem really be?
Instead of seeking a single motion planner $s$ for all of $X\times X$, we seek to find the minimal number $k$ of open subsets $U_{i}\subset X\times X$ covering $X$, such that there are continuous sections
\[
s_{i}:U_{i}\to X^{I},
\]
such that $\pi\circ s_{i}=\text{id}_{X\times X}|_{U_{i}}$.
This number $k$ is the \emph{topological complexity of $X$}, and we say $\textbf{TC}(X)=k$.
If no such $k$ exists, then we set $\textbf{TC}(X)=\infty$.

\begin{exam}
The configuration space of a robot arm consisting of $n$ bars $L_{1}, \dots, L_{n}$, such that $L_{i}$ and $L_{i+1}$ are connected a flexible joint can be viewed as the $n$-fold product of spheres, $S^{1}$s if the joints are hinges with one dimension of freedom and $S^{2}$s if the joints are balls and sockets with two dimensions of freedom.
In the hinge case the topology complexity is $n+1$, whereas in the ball and socket case the topological complexity is $2n+1$ \cite[Theorem 12]{farber2003topological}.
\end{exam}

\begin{exam}
If we shrink our robots from disks to points, we move from $\text{conf}(n,w)$ to $F_{n}(\R^{2})$, the ordered configuration space of $n$ points in the plane.
Farber and Yuzvinksy proved that $\textbf{TC}\big(F_{n}(\R^{2})\big)=2n-2$ \cite[Theorem 1]{farber2002topological}, and Farber and Grant extended this to ordered configurations in arbitrary $\R^{m}$, proving that $\textbf{TC}\big(F_{n}(\R^{m})\big)=2n-1$ is $m$ is even and $\textbf{TC}\big(F_{n}(\R^{m})\big)=2n-2$ if $m$ is odd \cite[Theorem 1]{farber2009topological}.
\end{exam}

Often the space $X$ is hard to work with directly, and this makes computing $\textbf{TC}(X)$ a challenge.
Fortunately, Farber proved that topological complexity is a homotopy invariant of $X$, \cite[Theorem 3]{farber2003topological}.
This result will allow us to replace $\text{conf}(n,w)$ with a finite cellular complex called $\text{cell}(n,w)$ that is significantly easier to manipulate.

\begin{prop}\label{TC is homotopy invariant}
(Farber \cite[Theorem 3]{farber2003topological}) Topological complexity is a homotopy invariant.
\end{prop}

Even though we can often replace $X$ with a simpler space $Y$, finding open sets $\{V_{1}, \cdots, V_{k}\}$ of $Y\times Y$ with continuous sections $s_{i}:V_{i}\to Y^{I}$ with $k$ relatively small is a challenge.
Proving that such a $k$ is minimal is even harder.
As such we seek upper and lower bounds for topological complexity.
Farber gave such bounds in \cite[Sections 3 and 4]{farber2003topological}, and we will show that they are sharp in the case of $\text{conf}(n,w)$ when $n>w$.

\subsection{An upper bound for topological complexity}

We seek an upper bound for the topological complexity of $X$.
As long as $X$ is path connected and paracompact such a bound arises from the dimension of $X$.

\begin{prop}\label{upper bound}
(Farber \cite[Theorem 4]{farber2003topological})
For any path-connected paracompact space $X$,
\[
\textbf{TC}(X)\le 2\dim X+1.
\]
\end{prop}

We may naturally view $\text{conf}(n,w)$ a subspace of $\R^{2n}$, proving that 
\[
\textbf{TC}\big(\text{conf}(n,w)\big)\le 4n+1.
\]

This bound is sub-optimal. 
As we will see in the next section, $\text{conf}(n,w)$ is homotopy equivalent to an $\big(n-\big\lceil\frac{n}{w}\big\rceil\big)$-dimensional cellular complex.
It follows that Proposition \ref{TC is homotopy invariant} proves that $\textbf{TC}\big(\text{conf}(n,w)\big)$ is at most $2n-2\big\lceil\frac{n}{w}\big\rceil+1$.
We will show that this bound is sharp when $w<n$.

\subsection{A lower bound for topological complexity}
Finding meaningful lower bounds for the topological complexity of a space, $X$, is a harder problem, as this corresponds to showing that \emph{any} motion planning program must consider a certain number of cases. 
Fortunately, the cup product structure on the cohomology ring of $X$ leads to a reasonable lower bound that will be sharp in our case.

Given a space $X$, let $\Delta_{X}\subset X\times X$ denote the diagonal, i.e., the set of points $\big\{(x, x)\in X\times X\big\}$.
A cohomology class 
\[
\phi\in H^{*}(X\times X;R)
\]
is called a \emph{zero-divisor} if its restriction to the diagonal is trivial, i.e.,
\[
\phi|_{\Delta_{X}}=0\in H^{*}(X;R).
\]
We say that the \emph{zero-divisor-cup-length} of $H^{*}(X;R)$ is the length of the longest non-trivial cup product $\big($in $H^{*}(X\times X;R)\big)$ of zero-divisors.
Farber proved that the zero-divisor-cup-length serves as a lower bound for the topological complexity of $X$.

\begin{prop}\label{cup product lower bound general}
(Farber \cite[Theorem 7]{farber2003topological})
The zero-divisor-cup-length of $H^{*}(X;R)$ is less than the topological complexity of $X$.
\end{prop}

Farber also noted that if $\alpha\in H^{*}(X;R)$, then
\[
\overline{\alpha}:=1\otimes \alpha-\alpha\otimes 1\in H^{*}(X;R)\otimes H^{*}(X;R)\cong H^{*}(X\times X;R)
\]
is a zero-divisor, where the isomorphism $H^{*}(X;R)\otimes H^{*}(X;R)\cong H^{*}(X\times X;R)$ is given by the K\"{u}nneth theorem.
We will use this construction to find a family of $2n-2\big\lceil\frac{n}{w}\big\rceil$ zero-divisors in $H^{*}\big(\text{conf}(n,w)\times \text{conf}(n,w)\big)$ that have non-trivial cup product, yielding a lower bound for the topological complexity of $\text{conf}(n,w)$.

\section{Bounds for $\textbf{TC}\big(\text{conf}(n,w)\big)$}
In this section we determine sharp upper and lower bounds for the topological complexity of $\text{conf}(n,w)$.
We do this by recalling the definitions of two spaces homotopy equivalent to $\text{conf}(n,w)$ that will allow us to use Propositions \ref{upper bound} and \ref{cup product lower bound general} to bound the topological complexity of $\text{conf}(n,w)$.
The first space is the cellular complex $\text{cell}(n,w)$, whose dimension can easily be determined, yielding an upper bound for $\textbf{TC}\big(\text{conf}(n,w)\big)$.
The second space, denoted $M(n,w)$, will allow us to compute certain cup products in $H^{*}\big(\text{conf}(n,w)\big)$, which we will use to find a lower bound for  $\textbf{TC}\big(\text{conf}(n,w)\big)$.
We begin by stating several facts about $\text{cell}(n,w)$.
In the next subsection we will recall Alpert and Manin's construction of a discrete gradient vector field on $\text{cell}(n,w)$, the critical cells of which correspond to a basis for the cohomology of $M(n,w)$.

\subsection{$\text{cell}(n,w)$ and an upper bound for $\textbf{TC}\big(\text{conf}(n,2)\big)$}

In this subsection we recall the definition of the cellular complex $\text{cell}(n,w)$,  which was introduced in \cite{blagojevic2014convex}, and is homotopy equivalent to $\text{conf}(n,w)$.

Let \emph{$\text{cell}(n)$} denote the cellular complex whose cells are represented by \emph{symbols} consisting of an ordering of the numbers $1, \dots, n$ separated by vertical bars into \emph{blocks} such that no block is empty. 
A cell $f\in \text{cell}(n)$ is a codimensional-$1$ \emph{face} of a cell $g\in \text{cell}(n)$ if $g$ can be obtained by deleting a bar in $f$ and shuffling the resulting block.
We say that $g$ is a \emph{co-face} of $f$.
It follows that cells represented by symbols with no bars are the top-dimensional cells of $\text{cell}(n)$.
It follows that we can view $\text{cell}(n)$ as an $(n-1)$-dimensional complex, and each bar in a symbol lowers the dimension of the corresponding cell by $1$.

By restricting how big a block can be in $\text{cell}(n)$, one gets the cellular complex $\text{cell}(n,w)$.
The cellular complex \emph{$\text{cell}(n,w)$} is the subcomplex of $\text{cell}(n)$ consisting of the cells represented by symbols whose blocks have at most $w$ elements; see Figure \ref{critical_cells_of_cell32}.
Alpert, Kahle, and MacPherson proved that $\text{cell}(n,w)$ is homotopy equivalent to $\text{conf}(n,w)$ \cite[Theorem 3.1]{alpert2021configuration}, allowing us to study it in place of $\text{conf}(n,w)$.

\begin{prop}
(Alpert--Kahle--MacPherson \cite[Theorem 3.1]{alpert2021configuration}) There is an $S_{n}$-equivariant homotopy equivalence $\text{conf}(n,w)\simeq \text{cell}(n,w)$.
\end{prop}

It follows from Proposition \ref{TC is homotopy invariant} that if we can determine the dimension of $\text{cell}(n,w)$ we will get an upper bound for the topological complexity of $\text{conf}(n,w)$.

\begin{prop}\label{dimension of cell(n,w)}
The dimension of $\text{cell}(n,w)$ is $n-\big\lceil\frac{n}{w}\big\rceil$.
\end{prop}

\begin{proof}
We seek to find the cells represented by the symbols with the least number of vertical bars.
The symbol 
\[
1\,\cdots\, w \Big| w+1\, \cdots\, 2w \Big|\cdots\Big| \Big\lceil\frac{n}{w}\Big\rceil-n+1\,\cdots\, n
\]
corresponds to a top dimensional cell of $\text{cell}(n,w)$ as it has the fewest possible vertical bars, as any fewer bars would result in a block of size greater than $w$ and every block, perhaps other than the last, has $w$ elements.
This cell is $\big(n-\big\lceil\frac{n}{w}\big\rceil\big)$-dimensional as it has $\big\lceil\frac{n}{w}\big\rceil-1$ bars.
\end{proof}

This gives improves our upper bound on the topological complexity of $\text{conf}(n,w)$.

\begin{prop}\label{upper bound for top complex of conf(n,w)}
\[
\textbf{TC}\big(\text{conf}(n,w)\big)\le 2n-2\Big\lceil\frac{n}{w}\Big\rceil+1.
\]
\end{prop}

\begin{proof}
This follows immediately from Propositions \ref{upper bound} and \ref{dimension of cell(n,w)}.
\end{proof}

Next, we find a lower bound the topological complexity of $\text{conf}(n,w)$.
Our bound will coincide with that of Proposition \ref{upper bound for top complex of conf(n,w)}, proving that they are sharp.

\subsection{$M(n,w)$ and a lower bound for $\textbf{TC}\big(\text{conf}(n,w)\big)$}
Having found an upper bound for $\textbf{TC}\big(\text{conf}(n,w)\big)$ we seek a suitable lower bound, i.e., the minimal number of cases any program controlling the movement of the $n$ robots about the aisle must contain. 
In order to use Proposition \ref{cup product lower bound general} we need to understand the cohomology ring of $\text{conf}(n,w)$.
Alpert and Manin construct a discrete gradient vector field on $\text{cell}(n,w)$. 
The critical $j$-cells of this vector field will correspond to a basis for $H^{j}\big(\text{conf}(n,w)\big)$, via Poincar\'{e}--Lefschetz duality applied to another space called $M(n,w)$ that is homotopy equivalent to $\text{conf}(n,w)$.

We begin by describing the critical cells of Alpert and Manin's discrete gradient vector field as they will correspond to elements in the cohomology of $\text{conf}(n,w)$; see \cite[Sections 3 and 4]{alpert2021configuration1} for further details.
Given a cell $f\in \text{cell}(n,w)$, we break $f$ into its blocks.
Each block is further decomposed into \emph{wheels}, where each entry of a block is the \emph{axle} of a wheel if it is the largest entry of the block up to that point, and the wheel consists of the axle and all the following smaller entries before the next axle.
A block $f_{i}$ of $f$ is a \emph{unicycle} if it consists of a single wheel, i.e., the largest element in $f_{i}$ is also the first element.
We say that $f_{i+1}$ is a \emph{follower} if $f_{i}$ is a unicycle (its \emph{leader}), which is not itself a follower and whose wheel is smaller than any of the follower's wheels.
The cell $f$ is critical if every block $f_{i}$ is either 
\begin{itemize}
\item a unicycle which is not a follower, or
\item a follower such that it and its leader have at least $w+1$ elements in total.
\end{itemize}
See Figure \ref{critical_cells_of_cell32} for an example.

\begin{figure}[h]
\centering
\captionsetup{width=.8\linewidth}
\includegraphics[width = 8cm]{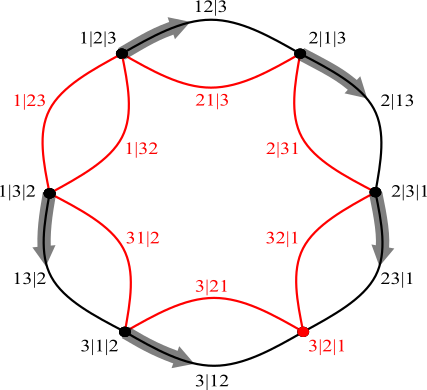}
\caption{The cellular complex $\text{cell}(3,2)$ with Alpert and Manin's discrete gradient vector field.
The critical cells of this vector field are in red.
This vector field arises from a total ordering of the cells of $\text{cell}(3,2)$, and a vector $[f, g]$ is included if and only if $f$ is the greatest face of $g$ and $g$ is the least coface of $f$.
For more on Alpert and Manin's ordering see \cite[Sections 3 and 4]{alpert2021configuration1}.
}
\label{critical_cells_of_cell32}
\end{figure}

We relate these critical cells to the cohomology ring of $\text{conf}(n,w)$ by studying another space, $M(n,w)$, that will allow us to take advantage of Poincar\'{e}--Lefschetz duality.
The \emph{follower-free} critical cells of $\text{cell}(n,w)$, i.e., those such that no block is a follower, will be of special interest to us, as the corresponding basis elements behave especially nicely under taking cup products.

Let $\emph{M(n,w)}\subset \R^{2n}$ be the configuration space of open disks of radius $1$ in a strip of any finite length $N>n$ and width $w+\epsilon$ for any $0<\epsilon<1$.
This space is a manifold (with boundary) of dimension $2n$, and it homotopy equivalent to $\text{conf}(n,w)$.
It follows from Poincar\'{e}--Lefschetz duality that
\[
H^{i}\big(\text{conf}(n,w)\big)\cong H^{i}\big(M(n,w)\big)\cong H_{2n-i}\big(M(n,w), \partial M(n,w)\big),
\]
and the cup product between classes in $H^{i}\big(M(n,w)\big)$ and $H^{j}\big(M(n,w)\big)$ is given by the transverse intersection map
\[
\pitchfork H_{2n-i}\big(M(n,w), \partial M(n,w)\big)\otimes H_{2n-j}\big(M(n,w), \partial M(n,w)\big)\to H_{2n-i-j}\big(M(n,w), \partial M(n,w)\big).
\]

Alpert and Manin describe a basis for $H_{2n-j}\big(M(n,w), \partial M(n,w)\big)$, hence a basis for $H^{j}\big(\text{conf}(n,w)\big)$, by associating basis elements to critical $j$-cells of $\text{cell}(n,w)$.
Let $f$ be the label of a critical cell, then the submanifold $V(f)\subset M(n,w)$ is the set of disk configurations such that
\begin{enumerate}
\item The disks in each block of $f$ are lined up vertically in order.
\item If a block $b_{1}$ comes before $b_{2}$ in $f$, they have a least $w+1$ elements combined, and one of them is a follower, then the column of disks labeled by elements of $b_{1}$ is to the left of the column of disks labeled by elements of $b_{2}$.
\end{enumerate}
See Figure \ref{pointinM144} for an example.

\begin{figure}[h]
\centering
\captionsetup{width=.8\linewidth}
\includegraphics[width = 12cm]{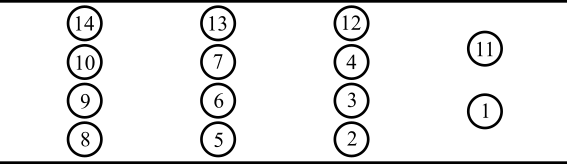}
\caption{A point in $V(14\, 10\, 9\, 8|13\, 7\, 6\, 5|12\, 4\,3\, 2|11\,1)$.
Note that the disks may wiggle up and down while remaining aligned.
The columns of the disks are allowed to move left and right, and in this submanifold the order of the columns can be permuted, so this point is one of the twenty four components of $V(14\, 10\, 9\, 8|13\, 7\, 6\, 5|12\, 4\,3\, 2|11\,1)$.
}
\label{pointinM144}
\end{figure}

Clearly this is a $(2n-j)$-dimensional submanifold of $M(n,w)$, and one can check that $\partial V(f)=V(f)\cap \partial M(n,w)$, proving that it represents a class in $H_{2n-j}\big(M(n,w), \partial M(n,w)\big)$.
Lemma 6.1 of Alpert and Manin proves that the $V(f)$ are in fact a basis for $H_{2n-j}\big(M(n,w), \partial M(n,w)\big)$ \cite[Lemma 6.1]{alpert2021configuration1}, and we will write $\nu(f)$ for the Poincar\'{e}--Lefschetz dual of $V(f)$.
Additionally, note that the follower-free critical cells correspond to submanifolds in which the columns of disks can be arranged in any order.

Given Proposition \ref{cup product lower bound general}, we wish to understand the cup product structure of $H^{*}\big(\text{conf}(n,w)\big)$.
First we recall a result of Alpert and Manin that will simplify our computations.

\begin{prop}\label{no repeats in cup product}
(Alpert--Manin \cite[Proposition 6.2]{alpert2021configuration1}) If there is a pair of labels $i$ and $j$ which are contained in the same block of critical cells $f$ and $g$, then $\nu(f)\cup \nu(g)=0$.
\end{prop}

Next, we note a fact that will greatly reduce the number of cup products we will need to take to determine whether a cup product is nonzero.
This result, along with the previous one, will allow us to use the identification $H^{2n-2\big\lceil\frac{n}{w}\big\rceil}\big(\text{conf}(n,w)\times\text{conf}(n,w)\big)\cong H^{n-\big\lceil\frac{n}{w}\big\rceil}\big(\text{conf}(n,w)\big)\otimes H^{n-\big\lceil\frac{n}{w}\big\rceil}\big(\text{conf}(n,w)\big)$.

\begin{prop}\label{product of more than w-1 terms with same first is trivial}
For $k\ge w$ and $i>j_{l}$ where $1\le l\le k$, let
\[
a_{l}=i\, j_{l}|n|\cdots|\widehat{i}|\cdots|\widehat{j_{l}}|\cdots|1
\]
be a follower-free critical cell of $\text{cell}(n,w)$.
Then,
\[
\prod\nu(a_{l})=0.
\]
\end{prop}

\begin{proof}
Consider the intersection of the $V(a_{l})$. 
In such an intersection the disk labeled $i$ is aligned above the disks labeled $j_{1},\dots, j_{k}$.
Since $k\ge w$, this means that $k+1>w$ disks must be aligned vertically in the intersection, which is impossible. 
Therefore, the intersection is empty and the cup product is $0$.
\end{proof}

The following two propositions will allow us to determine what terms might appear when we take the cup product of $n-\big\lceil\frac{n}{w}\big\rceil$ classes in $H^{1}\big(\text{conf}(n,w)\big)$.

\begin{prop}\label{product of critical cells with same first element is a sum of critical cells with the same first element}
Let
\[
f=i\, i_{1}\, \cdots\, i_{k}|n|\cdots|\widehat{i}|\cdots|\widehat{i_{l}}|\cdots|1
\]
be a follower-free critical $k$-cell of $\text{cell}(n,w)$ such that $k+1<w$, and let
\[
a=i\, j|n|\cdots|\widehat{i}|\cdots|\widehat{j}|\cdots|1
\]
be a follower-free critical $1$-cell of $\text{cell}(n,w)$ such that $j\neq i_{1},\dots, i_{k}$.
Then,
\[
\nu(f)\cup\nu(a)\sum\nu(e_{l}),
\]
where for $0\le l\le k$,
\[
e_{l}=i\, i_{1}\, \cdots\,i_{l}\, j\, i_{l+1}\, \cdots\,  i_{k}|n|\cdots|\widehat{i}|\cdots|\widehat{i_{l}}|\cdots|1.
\]
is a follower-free critical $(k+1)$-cell of $\text{cell}(n,w)$.
\end{prop}

\begin{proof}
The intersection of $V(f)$ and $V(a)$ is transverse.
It consists of all configurations such that the disk labeled $i$ is aligned above the disks labeled $j, i_{1},\dots, i_{k}$, and the disk labeled $i_{l}$ is above the disk labeled $i_{l+1}$. 
This exactly the submanifold $\sqcup V(e_{l})$, completing the proof.
\end{proof}

\begin{prop}\label{product of disjoint critical cells}
Let
\[
f=i_{1,1}\, \cdots\, i_{1, n_{1}}|\cdots|i_{j, 1}\, \cdots\, i_{j, n_{j}}|i_{n_{1}+\cdots+n_{j}+1}|\cdots|i_{n}
\]
be a follower-free critical cell of $\text{cell}(n,w)$,
and let
\[
g=i_{j+1, 1}\, \cdots\, i_{j+1, n_{j+1}}|i'_{n_{j+1}}|\cdots|i'_{n}
\]
be a follower-free critical cell of $\text{cell}(n,w)$, such that $i_{j+1, k}\in \{i_{n_{1}+\cdots+n_{j}+1},\dots,i_{n}\}$ for $1\le k\le n_{j+1}$.
Then
\[
\nu(f)\cup \nu(g)=\nu(e),
\]
where
\[
e=i_{1,1}\, \cdots\, i_{1, n_{1}}|\cdots|i_{j+1, 1}\, \cdots\, i_{j+1, n_{j+1}}|\cdots|i_{j, 1}\, \cdots\, i_{j, n_{j}}|i_{n_{1}+\cdots+n_{j}+1}|\cdots|\widehat{i_{j+1, 1}}|\cdots|\widehat{i_{j+1, n_{j+1}}}|\cdots|i_{n}
\]
is a follower-free critical cell of $\text{cell}(n,w)$, i.e., there is an $1\le l\le j$ such that $n_{l}\le n_{j+1}\le n_{l+1}$ and if $n_{l}=n_{j+1}$, then $i_{1, l}>i_{1,j+1}$, and if $n_{j+1}=n_{l+1}$, then $i_{1, j+1}>i_{1,l+1}$.
\end{prop}

\begin{proof}
It is clear that $V(f)$ and $V(g)$ have transverse intersection.
Moreover, their intersection is the submanifold of $V(f)$ such that the disks $i_{j+1, 1}, \dots, i_{j+1, n_{j+1}}$ are aligned vertically in that order.
It follows that this column of disks can be anywhere among the columns of the other disks of $V(f)$. 
This is exactly $V(e)$.
\end{proof}

Now we are ready to calculate the zero-divisor-cup-length of $H^{*}\big(\text{conf}(n,w)\big)$ for $n>w$.
Our proof relies on the observation that the submanifold of $M(n,w)$ corresponding to $V(f)$ where $f$ is a follower-free top dimensional critical cell of $\text{cell}(n,w)$ such that the $\big\lceil\frac{n}{w}\big\rceil$ blocks of $f$ have first element in $\big\{n,\dots, n-\big\lceil\frac{n}{w}\big\rceil+1\big\}$, see Figure \ref{pointinM144}, is a submanifold contained in the intersection of $n-\big\lceil\frac{n}{w}\big\rceil$ submanifolds arising from follower-free critical $1$-cells of $\text{cell}(n,w)$ such that the first element of the block with two elements is in $\big\{n,\dots, n-\big\lceil\frac{n}{w}\big\rceil+1\big\}$, and the second element is in $\big\{n-\big\lceil\frac{n}{w}\big\rceil,\dots, 1\big\}$.
If we carefully choose two disjoint sets $A$ and $B$ of $n-\big\lceil\frac{n}{w}\big\rceil$ such follower-free critical $1$-cells, we will get a non-trivial product in $H^{*}\big(\text{conf}(n,w)\otimes\text{conf}(n,w)\big)$.

\begin{lem}\label{cup divisor cup length for all w}
For $n>w$ the zero-divisor-cup-length of $H^{*}\big(\text{conf}(n,w)\big)$ is $2n-2\big\lceil\frac{n}{w}\big\rceil$.
\end{lem}

\begin{proof}
Set $m:=\big\lceil\frac{n}{w}\big\rceil$.
For $0\le i\le m-2$ and $0\le j\le w-2$, let
\[
a_{i, j}=(n-i)\, \big(n-m-(w-1)i-j\big)|n|\cdots|\widehat{(n-i)}|\cdots|\widehat{\big(n-m-(w-1)i-j\big)}|\cdots|1
\]
and, for $i=m-1$ and $0\le j\le (n\text{ mod}w)-2$, let
\[
a_{i,j}=(n-m+1)\, \big(n-m-(w-1)(m-1)-j\big)|n|\cdots|\widehat{(n-m+1)}|\cdots|\widehat{\big(n-m-(w-1)(m-1)-j\big)}|\cdots|1
\]
be critical $1$-cells of $\text{cell}(n,w)$.
Similarly, for $0\le i\le m-2$ and $0\le j\le w-2$, let
\[
b_{i, j}=(n-i+1)\, \big(n-m-(w-1)i-j\big)|n|\cdots|\widehat{(n-i)}|\cdots|\widehat{\big(n-m-(w-1)i-j\big)}|\cdots|1
\]
and, for $i=m-1$ and $0\le j\le (n\text{ mod}w)-2$,
\[
b_{i,j}=(n)\, \big(n-m-(w-1)(m-1)-j\big)|n|\cdots|\widehat{(n-m+1)}|\cdots|\widehat{\big(n-m-(w-1)(m-1)-j\big)}|\cdots|1
\]
be critical $1$-cells of $\text{cell}(n,w)$.
There are precisely $n-\big\lceil\frac{n}{w}\big\rceil$ of each of the $a_{i, j}$s and $b_{i,j}$s.

Set
\[
\alpha_{i,j}=\nu(a_{i,j})\indent\text{and}\indent\beta_{i,j}=\nu(b_{i,j}),
\]
and
\[
\overline{\alpha_{i,j}}=\alpha_{i,j}\otimes 1-1\otimes \alpha_{i,j}\indent\text{and}\indent
\overline{\beta_{i,j}}=\beta_{i,j}\otimes 1-1\otimes \beta_{i,j}.
\]

We claim that 
\[
\prod\overline{\alpha_{i,j}}\cup\overline{\beta_{i,j}}\neq0.
\]

To see this note that any nonzero term in the cup product in $H^{2n-2\big\lceil\frac{n}{w}\big\rceil}\big(\text{conf}(n,w)\times \text{conf}(n,w)\big)$ must be in $H^{n-\lceil\frac{n}{w}\rceil}\big(\text{conf}(n,w)\big)\otimes H^{n-\lceil\frac{n}{w}\rceil}\big(\text{conf}(n,w)\big)$ as $H^{k}\big(\text{conf}(n,w)\big)=0$ for $k>n-\big\lceil\frac{n}{w}\big\rceil$.

Consider any cup product of a total of $n-\big\lceil\frac{n}{w}\big\rceil$ of the $\alpha_{i,j}$s and $\beta_{i,j}$s.
Since the cohomology ring is graded commutative, it follows that we may first take the product of all the terms such that $i=0$, take the product of all the terms such that $i=1$, etc, and then take the product of the resulting classes.
Note that by Proposition \ref{product of more than w-1 terms with same first is trivial}, if for any $0\le i\le m-1$ there are more than $w-1$ total terms of the form $\alpha_{i,j}$ or $\beta_{i,j}$, then the cup product is $0$. 
By Proposition \ref{product of critical cells with same first element is a sum of critical cells with the same first element} for each  $0\le i\le m-1$ we may rewrite the product of the $\alpha_{i, j}$s and $\beta_{i,j}$s as a sum of products of $\nu(e_{i})$, where the $e_{i}$ are follower-free critical cells of the form
\[
e_{i}=n-i\, i_{i,2}\, \cdots, i_{i,l}|n|\widehat{n-i}|\cdots|\widehat{i_{i,2}}|\cdots|\widehat{i_{i,l}}|\cdots|1,
\]
where the $i_{2},\dots, i_{l}$ are permutations of the second elements in the block with two elements of the critical cells of the form $a_{i,j}$ and $b_{i,j}$ corresponding to the $\alpha_{i, j}$s and $\beta_{i,j}$s.
Moreover, if some element appears as the second element in the block with two elements in two of the $a_{i,j}$s and $b_{i,j}$s, then the cup product is trivial.
This follows from the fact the first element of the two blocks must be different, and since we may assume that no first entry $n-i$ appears more than $w-1$ times in product, any two first elements of a block must appear more than $w-1$ times in the product.
It follows that every non-trivial product of $n-\big\lceil\frac{n}{w}\big\rceil$ of the $\alpha_{i,j}$s and $\beta_{i,j}$s can be written as sum of products the $\nu(e_{i})$s where for each product of the $\alpha_{i,j}$s and $\beta_{i,j}$s and at least $1$ of $i$, the labels in the first block of the $e_{i}$s is different than in any other product.
By Proposition \ref{product of disjoint critical cells} every such product can be written as a sum of follower-free critical $\big(n-\big\lceil\frac{n}{w}\big\rceil\big)$-cells of $\text{cell}(n,w)$ whose $m$ blocks are the first blocks of the $e_{i}$s for $0\le i\le m-1$.
Additionally, these are all basis elements.
Moreover, every nontrivial product of the $\alpha_{i, j}$s and $\beta_{i,j}$s is such that the elements in at least one of the $m$ blocks are unique to that set of $n-\big\lceil\frac{n}{w}\big\rceil$ of the $\alpha_{i,j}$s and $\beta_{i,j}$s.
Therefore, $\prod\overline{\alpha_{i,j}}\cup\overline{\beta_{i,j}}\neq 0$.
\end{proof}

Combining this lemma and Proposition \ref{cup product lower bound general} yields the following lower bound for $\textbf{TC}\big(\text{conf}(n,w)\big)$.

\begin{cor}\label{lower bound for tcconfnw}
\[
\textbf{TC}\big(\text{conf}(n,w)\big)>2n-2\Big\lceil\frac{n}{w}\Big\rceil.
\]
\end{cor}

It follows that we have a proof of Theorem \ref{top complexity of conf(n,w)}, which we restate for convenience.

\begin{T1}
  \thmtext
\end{T1} 

\begin{proof}
If $n=1$, then $\text{conf}(n,w)$ is contractible, so \cite[Theorem 1]{farber2003topological} proves that $\textbf{TC}\big(\text{conf}(1,w)\big)=1$.
If $1<n\le w$, then $\text{conf}(n,w)$ is homotopy equivalent to $F_{n}(\R^{2})$, so \cite[Theorem 1]{farber2002topological} proves that $\textbf{TC}\big(\text{conf}(n,w)\big)=2n-2$.
Finally, if $n>w$, then Corollary \ref{lower bound for tcconfnw} proves that $\textbf{TC}\big(\text{conf}(n,w)\big)>2n-2\big\lceil\frac{n}{w}\big\rceil$ and Proposition \ref{upper bound for top complex of conf(n,w)} proves that $\textbf{TC}\big(\text{conf}(n,w)\big)\le2n-2\big\lceil\frac{n}{w}\big\rceil+1$, forcing
\[
\textbf{TC}\big(\text{conf}(n,w)\big)=2n-2\Big\lceil\frac{n}{w}\Big\rceil+1.
\]
\end{proof}

It follows that any program guiding $n$ robots in the width $w$ aisle must consider at least $2n-2\big\lceil\frac{n}{w}\big\rceil+1$ cases.
Moreover, if we are very clever, we could develop a program that considers exactly $2n-2\big\lceil\frac{n}{w}\big\rceil+1$ cases.

\bibliographystyle{amsalpha}
\bibliography{TopComplexConfnwBib}
\end{document}